\documentclass[twoside,11pt]{article}
\usepackage{graphicx,amsmath,amsthm,amssymb,latexsym,amsfonts,color}
\usepackage[bookmarksnumbered, colorlinks, plainpages]{hyperref}
\newtheorem{thm}{\bf Theorem}
\newtheorem{lem}{\bf Lemma}

\newtheorem{example}{\bf Example}

\newtheorem{remark}{\bf Remark}

\begin{document}
\title{A block lower triangular preconditioner for a class of complex symmetric  system of linear equations}
\author{{ Davod Khojasteh Salkuyeh and Tahereh Salimi Siahkalaei}\\[2mm]
\textit{{\small Faculty of Mathematical Sciences, University of Guilan, Rasht, Iran}} \\
\textit{{\small E-mails: khojasteh@guilan.ac.ir, salimi-tahereh@phd.guilan.ac.ir}\textit{}}}
\date{}
\maketitle
\noindent{\bf Abstract.} We present a block lower triangular (BLT) preconditioner to accelerate the convergence of nthe Krylov subspace iterative methods, such as generalized minimal residual (GMRES), for solving a broad class of complex symmetric system of linear equations. We analyze the eigenvalues distribution of preconditioned coefficient matrix. Numerical experiments are given to demonstrate the effectiveness of the BLT preconditioner.
 \\[-3mm]

\noindent{\bf  AMS subject classifications}: 65F10, 65F50, 65F08, 65F10.\\

\noindent{\bf  Keywords}: {complex linear systems, block lower triangular, symmetric positive definite, preconditioning, GMRES, GSOR, MHSS.}\\
\noindent{\it   \\

\pagestyle{myheadings}\markboth{D.K. Salkuyeh,  and  T. Salimi}{BLT preconditioner for complex symmetric linear systems}

\thispagestyle{empty}

\section{Introduction} \label{Eq1}\rm
Consider the complex system of linear equations of the form
\begin{equation}\label{Eq1}
Au=b, 
\end{equation}
where 
\[
A=W+iT,\quad  W,T \in \mathbb{R }^{n \times n},
\]
in which $u=x+iy$ and $b=p+iq$, such that the vectors $x,y,p$ and $q$ are in $\mathbb{R}^{n}$ and $i=\sqrt{-1}$. We assume that $W$ and $T$ are symmetric positive semidefinite matrices such that at least one of them, e.g., $W$, is positive definite. Complex symmetric linear systems of this kind arise in many important problems in scientific computing and engineering applications. For example, FFT-based solution of certain time-dependent PDEs \cite{Bertaccini}, diffuse optical tomography \cite{Arridge}, algebraic eigenvalue problems \cite{Moro, Schmitt}, molecular scattering \cite{Poirier}, structural dynamics \cite{Feriani} and lattice quantum chromodynamics \cite{Frommer}.

In recent years, there have been many works to solve (\ref{Eq1}), and several iterative methods have been presented in literature. For example, based on the Hermitian and skew-Hermitian splitting (HSS) of the matrix $A$, Bai et al. in \cite{Bai1} introduced the Hermitian/skew-Hermitian splitting (HSS) method to solve positive definite system of linear equations. Next, Bai et al.  presented a modified version of the HSS iterative method say (MHSS) \cite{Bai2} to solve systems of the form (\ref{Eq1}). The matrix $A$ naturally possesses the Hermitian/Skew-Hermitian (HS) splitting
\begin{equation}\label{Eq2}
A=H+S,
\end{equation}
where
\[
H=\frac{1}{2}(A+A^H)=W, \quad S=\frac{1}{2}(A-A^H)=iT,
\]
with $A^H$ being the conjugate transpose of $A$. In this case, the HSS and MHSS methods to solve \eqref{Eq1} can be written as follows.

\medskip

\noindent {\bf The HSS method}: \verb"Let" $u^{(0)} \in  {\mathbb{C }^{n}}$ \verb"be an initial guess". \verb"For" $k=0,1,2,\ldots$, \verb"until" $\{u^{(k)}\}$ \verb"converges", \verb"compute" ${u^{(k+1)}}$  \verb"according to the" \verb"following sequence":
\begin{equation}\label{Eq3}
\begin{cases}
(\alpha I+W){u^{(k+\frac{1}{2})}}=(\alpha I-iT){u^{(k)}}+b, \\
(\alpha I+iT){u^{(k+1)}}=(\alpha I-W){u^{(k+\frac{1}{2})}}+b,
\end{cases}
\end{equation}
\verb"where" $\alpha$ \verb"is a given positive constant and" $I$ \verb"is the identity matrix".

\medskip

The HSS iteration can be reformulated into the standard form
$$
u^{(k+1)}=G_{\alpha}u^{(k)}+C_{\alpha}b, \quad  k=0,1,2,\ldots,
$$
where
$$
G_{\alpha}=(\alpha I+iT)^{-1}(\alpha I-W)(\alpha I+W)^{-1}(\alpha I-iT),
$$
and
$$
C_{\alpha}=2\alpha(\alpha I+iT)^{-1}(\alpha I+W)^{-1}.
$$
\noindent Assumming
$$
M_{\alpha}=\frac{1}{2\alpha}(\alpha I+W)(\alpha I+iT),  \quad   N_{\alpha}=\frac{1}{2\alpha}(\alpha I-W)(\alpha I-iT),
$$
it holds that
$$
A=M_{\alpha}-N_{\alpha}, \quad  G_{\alpha}={M_{\alpha}}^{-1}N_{\alpha}.
$$
Hence, it follows that the matrix $M_{\alpha}$ can be used as a preconditioner for the complex symmetric system \eqref{Eq1}. Note that the multiplicative factor $1/(2\alpha)$ has no effect on the preconditioned system and therefore it can be dropped. Hence, the HSS preconditioner can be considered as $P_{\alpha}=(\alpha I+W)(\alpha I+iT)$.

\medskip

\noindent {\bf The MHSS method}: \verb"Let" $u^{(0)} \in  {\mathbb{C }^{n}}$ \verb"be an initial guess". \verb"For" $k=0,1,2,\ldots$, until $\{u^{(k)}\}$ \verb"converges, compute" ${u^{(k+1)}}$  \verb"according to the following sequence":
\begin{equation}\label{Eq4}
\begin{cases}
(\alpha I+W){u^{(k+\frac{1}{2})}}=(\alpha I-iT){u^{(k)}}+b, \\
(\alpha I+T){u^{(k+1)}}=(\alpha I+iW){u^{(k+\frac{1}{2})}}-ib,
\end{cases}
\end{equation}
\verb"where" $\alpha$ \verb"is a given positive constant and" $I$ \verb"is the identity matrix".

\medskip

In \cite{Bai2}, it has been shown that if $W$ and $T$ are symmetric positive definite and symmetric positive semidefinite, respectively, then the MHSS iterative method is convergent.  Since $\alpha I+T$ and $\alpha I+W$ are symmetric positive definite, the involving sub-system in each step of the MHSS iteration can be solved exactly by using the Cholesky factorization of the coefficient matrices or inexactly by the conjugate gradient (CG) method. Similar to the HSS method, Eq. (\ref{Eq4}) can be written in the stationary form
$$
u^{(k+1)}=L_{\alpha}u^{(k)}+R_{\alpha}b, \quad  k=0,1,2,\ldots,
$$
where
$$
L_{\alpha}=(\alpha I+T)^{-1}(\alpha I+iW)(\alpha I+W)^{-1}(\alpha I-iT),
$$
and
$$
R_{\alpha}=(1-i)\alpha(\alpha I+T)^{-1}(\alpha I+W)^{-1}.
$$
Then, $A=F_{\alpha}-H_{\alpha}$ and $L_{\alpha}={F_{\alpha}}^{-1}H_{\alpha}$, where
\[
F_{\alpha}=\frac{1+i}{2\alpha}(\alpha I+W)(\alpha I+T),\quad H_{\alpha}=\frac{1+i}{2\alpha}(\alpha I+iW)(\alpha I-iT).
\]
Therefore matrix $Q_{\alpha}=(\alpha I+W)(\alpha I+T)$ can be used as the MHSS preconditioner.

It is possible to convert the complex system (\ref{Eq1}) to the real-valued form
\begin{equation}\label{Eq5}
\mathcal{A} u=
 \begin{bmatrix}
W & -T \\
T & W
\end{bmatrix}
 \begin{bmatrix}
x \\
y
 \end{bmatrix}
 =
\begin{bmatrix}
p \\
q
\end{bmatrix}.
\end{equation}

\noindent Under our hypotheses, it can be easily proved that the matrix $\mathcal{A}$ is nonsingular. Recently, Salkuyeh et al. in \cite{Salkuyeh1} solved system (\ref{Eq5}) by the generalized successive overrelaxation (GSOR) iterative method. They split the coefficient matrix of the system (\ref{Eq5}) as
$$
 \mathcal{A}=\mathcal{D}-\mathcal{L}-\mathcal{U},
 $$
 where
 $$
 \mathcal{D}=
\begin{bmatrix}
W & 0 \\
0 & W
\end{bmatrix},\quad \mathcal{L}=
\begin{bmatrix}
0 & 0 \\
-T & 0
\end{bmatrix}, \quad \mathcal{U}=
\begin{bmatrix}
0 & T \\
0 & 0
\end{bmatrix}.
$$
So, for $0\neq \alpha \in \mathbb{R}$, they construct GSOR method  as
$$
\begin{bmatrix}
x^{k+1}\\
y^{k+1}
\end{bmatrix}=\mathcal{G}_{\alpha}
\begin{bmatrix}
x^{k}\\
y^{k}
\end{bmatrix}+\mathcal{C}_{\alpha}
\begin{bmatrix}
p\\
q
\end{bmatrix},
$$
where
$$
\mathcal{G}_{\alpha}={(\mathcal{D}-\alpha \mathcal{L})}^{-1}((1-\alpha)\mathcal{D}+\alpha \mathcal{U})=
\begin{bmatrix}
W & 0\\
\alpha T & W
\end{bmatrix}^{-1}
\begin{bmatrix}
(1-\alpha)W  &  \alpha T\\
0                  &  (1-\alpha)W
\end{bmatrix},
$$
and
$$
\mathcal{C}_{\alpha}=\alpha(\mathcal{D}-\alpha \mathcal{L})^{-1}=\alpha
\begin{bmatrix}
W & 0\\
\alpha T & W
\end{bmatrix}^{-1}.
$$
After some simplifications the GSOR method can be written as following.

\medskip

\noindent {\bf The GSOR iteration method}: \verb"Let" $(x^{(0)}; y^{(0)}) \in  {\mathbb{R }^{n}}$ \verb"be an initial guess. For" $k=0,1,2,\ldots$, \verb"until" $\{(x^{(k)};y^{(k)})\}$ \verb"converges, compute" ${(x^{(k+1)};y^{(k+1)})}$  \verb"according to the" \\ \verb"following sequence"
\begin{equation}\label{Eq6}
\begin{cases}
W x^{(k+1)}=(1-\alpha)Wx^{(k)}+\alpha T y^{(k)}+\alpha p, \\
Wy^{(k+1)}=-\alpha T x^{(k+1)}+(1-\alpha)W y^{(k)}+ \alpha q,
\end{cases}
\end{equation}
\verb"where" $\alpha$ \verb"is a given positive constant and" $I$ \verb"is the identity matrix".

In \cite{Salkuyeh1}, it has been shown that if $W$ and $T$ are symmetric positive definite and symmetric, respectively, then the GSOR method is convergent. Letting
$$
\mathcal{M}_{\alpha}=\frac{1}{\alpha}(\mathcal{D}-\alpha \mathcal{L}), \quad \mathcal{N}_{\alpha}=\frac{1}{\alpha}((1-\alpha)\mathcal{D}+\alpha \mathcal{U}),
$$
it holds that
$$
\mathcal{A}=\mathcal{M}_{\alpha}-\mathcal{N}_{\alpha}, \quad \mathcal{G}_{\alpha}={\mathcal{M}_{\alpha}}^{-1}\mathcal{N}_{\alpha}.
$$
Hence
$$
\mathcal{M}_{\alpha}=\frac{1}{\alpha}(\mathcal{D}-\alpha \mathcal{L}),
$$
can be used as a preconditioner for the system \eqref{Eq5}. As usual the multiplicative factor $1/\alpha$ can be neglected.

In this paper we propose a block lower triangular (BLT) preconditioner for the system \eqref{Eq5} and investigate the eigenvalues distribution of the preconditioned coefficient matrix. Then the preconditioned system is solved by the restarted version of the GMRES (GMRES($m$)) method.

The rest of the paper is organized as follows. In Section \ref{Sec2} our BLT preconditioner is introduced and its properties are investigated. Section \ref{Sec3} is devoted to some numerical experiments to show the effectiveness of  BLT preconditioner. Finally, some concluding remarks are given  in Section \ref{Sec4}.

\section{The new block lower triangular preconditioner} \label{Sec2}

We introduce a preconditioner of the form
\begin{equation}\label{Eq7}
\mathcal{G}_{\alpha}=
\begin{bmatrix}
W & 0\\
\alpha I & W
\end{bmatrix},
\end{equation}
for system \eqref{Eq5} and propose using the Krylov subspace method such as GMRES, to accelerate the convergence of the iteration, where $\alpha>0$. It is known that for the SPD problems, the rate of convergence of CG depends on the distribution of the eigenvalues of the system coefficient matrix. For nonsymmetric problems the eigenvalues may not describe  the convergence of nonsymmetric matrix iterations like GMRES. Nevertheless, a clustered spectrum (away from 0) often results  in rapid convergence \cite{Benzi3,Cao,BaiCAM}. Therefore, we need to examine the eigenvalues distribution of the matrix $\mathcal{G}_{\alpha}^{-1}\mathcal{A}$.
\begin{lem}\label{Lem1}
Let $W\in  \mathbb{R }^{n \times n}$ be symmetric positive definite and $T\in \mathbb{R}^{n \times n}$ be symmetric positive semidefinite. Let also $\lambda_k$ be an eigenvalue of the preconditioned matrix $\mathcal{G}_{\alpha}^{-1} \mathcal{A}$, and $u_k=(x_k;y_k)$ be the corresponding eigenvector. If $y_k=0$, then $\lambda_k=1$.
\end{lem}
\begin{proof}
We have $\mathcal{G}_{\alpha}^{-1}\mathcal{A}u_k=\lambda_k u_k$, or
\[
\begin{bmatrix}
W & -T \\
T  & W
\end{bmatrix}
\begin{bmatrix}
x_k \\
y_k
\end{bmatrix}=\lambda_k
\begin{bmatrix}
W  & 0\\
\alpha I & W
\end{bmatrix}
\begin{bmatrix}
x_k \\
y_k
\end{bmatrix},
\]
which is itself  equivalent to
\begin{equation}\label{Eq8}
    \begin{cases}
      Wx_k-Ty_k=\lambda_k Wx_k, \\
      Tx_k+Wy_k=\alpha \lambda_k x_k+\lambda_k Wy_k.
    \end{cases}
\end{equation}
If $y_k=0$, then from the first equation of \eqref{Eq8} we get $(1-\lambda_k)Wx_k=0$. Since $W$ is symmetric positive definite, we see that $x_k=0$ or $\lambda_k=1$. If $x_k=0$ then $u_k=0$, which is a contradiction. Hence $\lambda_k=1$.
\end{proof}

\begin{lem}\label{Lem2}
Let $W\in  \mathbb{R }^{n \times n}$ and $T \in  \mathbb{R }^{n \times n}$ be symmetric positive definite and symmetric positive semidefinite, respectively, and $(\lambda_k,u_k=(x_k;y_k))$ be an eigenpair of the preconditioned matrix $\mathcal{G}_{\alpha}^{-1}\mathcal{A}$. For $y_k\neq 0$,  set \begin{equation}\label{Eq9}
  a_k=\frac{y_k^{*}Ty_k}{y_k^* y_k},\quad
  b_k=\frac{y_k^{*}TW^{-2}Ty_k}{y_k^* y_k},\quad
  c_k=\frac{y_k^{*}TW^{-1}T W^{-1}Ty_k}{y_k^* y_k}.
 \end{equation}
Then, $\lambda_k=1$ or
\begin{equation}\label{Eq10}
  \lambda_k-1=\frac{\alpha b_k \pm \sqrt{\Delta(\alpha)}}{2a_k},
  \end{equation}
where  $\Delta(\alpha)=\alpha^{2}b_k^{2}+4a_k(\alpha b_k-c_k)$.
\end{lem}
\begin{proof}
As we saw in Lemma \ref{Lem1}, $\mathcal{G}_{\alpha}^{-1}\mathcal{A}u_k=\lambda_ku_k$ is equivalent to (\ref{Eq8}). It follows from the first equation of \eqref{Eq8} that $(1-\lambda_k)Wx_k=Ty_k$. If $\lambda_k\neq 1$, then we get
\[
x_k=\frac{1}{1-\lambda_k} W^{-1}Ty_k.
\]
Substituting $x_k$ into the second equation of \eqref{Eq8} yields
$$
(1-\lambda_k)Wy_k=\alpha \frac{\lambda_k}{1-\lambda_k}W^{-1}Ty_k-\frac{1}{1-\lambda_k}TW^{-1}Ty_k.
$$
Simplifying this equation gives
\begin{equation}\label{Eq11}
(1-\lambda_k)^2 Wy_k=\alpha \lambda_k W^{-1}Ty_k-TW^{-1}Ty_k.
\end{equation}
According to Lemma \ref{Lem1} if $y_k=0$, then $\lambda_k=1$ and there is nothing to prove. Otherwise, multiplying both sides of \eqref{Eq11} by ${y_k}^{*}TW^{-1}$ results in
$$
(1-\lambda_k)^{2}  \frac{{y_k}^{*}Ty_k}{{y_k}^{*}y_k}=\alpha \lambda_k \frac{{y_k}^{*}TW^{-2}Ty_k}{{y_k}^{*}y_k}-\frac{{y_k}^{*}TW^{-1}TW^{-1}Ty_k}{{y_k}^{*}y_k}.
$$
Now from the latter equation and Eq. \eqref{Eq9}, we get
\begin{equation}\label{Eq12}
a_k(1-\lambda_k)^2=\alpha \lambda_k b_k-c_k.
\end{equation}
Obviously, both of the matrices $TW^{-2}T$ and $TW^{-1}T W^{-1}T$ are symmetric positive semidefinite. Therefore, we deduce that $a_k,b_k,c_k\geq 0$.  We consider the following two cases.
\begin{itemize}
\item If $a_k=0,$ then from \eqref{Eq9}, we have ${y_k}^{*}Ty_k=0$. Hence $Ty_k=0$ and therefore from the first equation in \eqref{Eq8} we obtain $\lambda_k=1$ or $x_k=0$. If $x_k=0$ then from the second equation in \eqref{Eq8}, $(1-\lambda_k)W y_k=0$. This implies that $\lambda_k=1$, since $W$ is nonsingular and $y_k \neq0$.

\item If $a_k \neq 0$, then by solving the quadratic equation \eqref{Eq12} we get
\[
\lambda_k-1=\frac{\alpha b_k \pm \sqrt{\Delta(\alpha)}}{2a_k},
\]
where  $\Delta(\alpha)=\alpha^{2}b_k^{2}+4a_k(\alpha b_k-c_k).$
\end{itemize}
 Therefore, the proof is complete.
\end{proof}
\rm \begin{thm}\label{Thm1}
Let $W$ and $T \in  \mathbb{R }^{n \times n}$ be symmetric positive definite and symmetric positive semidefinite, respectively. Let also
\begin{equation}\label{Eq13}
0<\alpha \leq \min\left\{\frac{2}{b_k}(\sqrt{a_k(a_k+c_k)}-a_k):~b_k \neq 0\right\}.
\end{equation}
Then the eigenvalues of the matrix $\mathcal{G}_{\alpha}^{-1} \mathcal{A}$ are enclosed in a circle of radius $\sqrt{{(c_k-\alpha b_k)}/{a_k}}$ centered at (1,0) where $a_k$, $b_k$ and $c_k$ were defined in \eqref{Eq9}.
\end{thm}
\begin{proof}
Let $\lambda_{k}$ be an eigenvalue of the matrix $\mathcal{G}^{-1} \mathcal{A}$. From Lemma \ref{Lem2}, we have $\lambda_k=1$ or
$\lambda_k-1=({\alpha b_k \pm \sqrt{\Delta(\alpha)}})/{2a_k}$. If  $\lambda_k=1$, there is nothing to prove. Otherwise, if $\Delta(\alpha)>0$ then the quadratic equation \eqref{Eq12} has two roots $\lambda_k^+$ and $\lambda_k^-$ which satisfy
\[
|\lambda_k^{+}-1|=\frac{|\alpha b_k + \sqrt{\Delta(\alpha)}|}{2a_k}=:r_1,\quad |\lambda_k^{-}-1|=\frac{|\alpha b_k - \sqrt{\Delta(\alpha)}|}{2a_k}=:r_2.
\]
In this case, we deduce that these eigenvalues are enclosed in a circle of radius $\max\{r_1,r_2\}=r_1$ centered at $1$.
On the other hand, if $\Delta(\alpha) \leq 0$, then from Eq. \eqref{Eq10}, we have
$$
\lambda_k^{\pm}-1=\frac{\alpha b_k \pm i\sqrt{-\alpha^{2}b_k^{2}+4a_k(c_k-\alpha b_k)}}{2a_k},
$$
which gives
\begin{equation}\label{Eq14}
\vert{\lambda_k^{\pm}-1}\vert=\sqrt{\frac{c_k-\alpha b_k}{a_k}}=:r.
\end{equation}
This means that the eigenvalues of the preconditioned matrix are enclosed in a circle of radius $r$ centered at $1$. It is not difficult to see that $r\leq r_1$. Therefore, to have a more clustered eigenvalues around $1$ we should choose the parameter $\alpha$ such that $\Delta(\alpha)\leq 0$.
We have $\Delta(\alpha)\leq 0 $ if and only if $\alpha\in[\alpha_1,\alpha_2]$, where
\[
\alpha_1=\frac{-2}{b_k}\left(a_k+\sqrt{a_k(a_k+c_k)}\right)\quad \textrm{and} \quad  \alpha_2=\frac{-2}{b_k}\left(a_k-\sqrt{a_k(a_k+c_k)}\right).
\]
It is necessary to mention that if $b_k=0$, then we have $\Delta(\alpha)=-4a_kc_k\leq 0$.
Since $a_k,b_k$ and $c_k$ are nonnegative, we deduce that $\alpha_1$ is nonpositive and $\alpha_2$ is nonnegative. Therefore, having in mind that $\alpha>0$ it is enough to choose $\alpha\in (0,\alpha_2]$.  Obviously, if we consider the above result over all the eigenpairs the desired result is obtained.
\end{proof}


\begin{thm}\rm
Let $W$ and $T \in  \mathbb{R }^{n \times n}$ be symmetric positive definite and symmetric positive semidefinite, respectively and $0 < \nu_1 \le \nu_2 \le \cdots \le \nu_n$ and $0 \le \mu_1 \le \mu_2 \le \cdots \le \mu_n \neq 0$ be the eigenvalues of $W$ and $T$, respectively. If $\alpha$ satisfies (\ref{Eq13}), then the eigenvalues of $\mathcal{G}^{-1}_{\alpha}\mathcal{A}$ lie in the following disk
\begin{equation}\label{Eq15}
\vert{\lambda-1}\vert \le \sqrt{\frac{\mu^3_n \nu^2_n-\alpha {\mu^2_1} {\nu^2_1}}{\nu^2_1 \nu^2_n  \mu_1}}.
\end{equation}
\end{thm}
\begin{proof} Since the matrix $T$ is symmetric, using the Courant-Fisher min–max theorem \cite{Axe}, we get
\begin{equation}\label{Eq16}
   \mu_1 \le a_k \le \mu_n.
\end{equation}
On the other hand,
\[
b_k=\frac{y_k^{*}TW^{-2}Ty_k}{y_k^* y_k}=\frac{{y_k}^{*}TW^{-2}Ty_k}{{y_k}^{*}T Ty_k}\frac{{y_k}^{*}T Ty_k}{{y_k}^{*}y_k}.
\]
Letting  $z_k=T y_{k}$, again by using the Courant-Fisher min–max theorem it follows that
\[
\lambda_{\min} (W^{-2}) \lambda_{\min} (T^2)  \le b_k=\frac{{z_k}^{*}W^{-2}z_k}{{z_k}^{*}z_k} \frac{{y_k}^{*}T^2 y_k}{{y_k}^{*}y_k} \le \lambda_{\max} (W^{-2}) \lambda_{\max} (T^2),
\]
where $\lambda_{\min} (\cdot)$ and $\lambda_{\max} (\cdot)$ stand for the smallest and largest eigenvalues of the matrix, respectively. Hence,
\[
\lambda^2_{\min} (W^{-1}) {\mu^2_1} \le b_k \le \lambda^2_{\max} (W^{-1}) {\mu^2_n},
\]
and therefore
\begin{equation}\label{Eq17}
(\frac{\mu_1}{\nu_n})^2 \le b_k \le (\frac{\mu_n}{\nu_1})^2.
\end{equation}
Similarly, it is easy to see that
\begin{equation}\label{Eq18}
   \frac{\mu^3_1}{\nu^2_n} \le c_k \le \frac{\mu^3_n}{\nu^2_1}.
\end{equation}
Now from Theorem \ref{Thm1} and Eqs. (\ref{Eq16}), (\ref{Eq17}) and (\ref{Eq18}) the desired result is obtained.
\end{proof}


\begin{remark}\rm
Assuming $\alpha_k=\frac{2}{b_k}(\sqrt{a_k(a_k+c_k)}-a_k)$, from Eqs. (\ref{Eq16}), (\ref{Eq17}), (\ref{Eq18}) we obtain
\[
\alpha_k=\frac{2}{b_k} \frac{a_k c_k}{\sqrt{a_k(a_k+c_k)}+a_k} \ge \frac{2\nu^3_1 \mu^4_1}{\nu^2_n \mu^3_n(\sqrt{\nu^2_1+\mu^2_n}+\nu_1)}=\alpha^*.
\]
for all $k$. Therefore, Eq. (\ref{Eq13}) can be replaced by $0<\alpha \le \alpha^{*}$.
\end{remark}
\begin{remark}\rm
If we choose
\[
\alpha=\tilde{\alpha}=\frac{\mu^3_n \nu^2_n}{\nu^2_1 \mu^2_1},
\]
then the right-hand side of Eq. (15) becomes zero. However, it may not belong to the interval $(0,\alpha^*]$.
Hence, we select the nearest $\alpha$ in interval $(0,\alpha^{*}]$ to $\tilde{\alpha}$.
\end{remark}

\section{Numerical experiments}\label{Sec3}

We use three test problems from \cite{Bai1} and an example of \cite{Yang}, to illustrate the feasibility and effectiveness of the BLT preconditioner when it is employed as a perconditioner for the GMRES method  to solve the real system \eqref{Eq5}. To do so, we compare the numerical results of the restarted generalized minimal residual (GMRES($m$)) method in conjunction with the BLT preconditioner with those of the MHSS and the GSOR preconditioners and the restarted GMRES method without preconditioning.
Numerical results are compared in terms of  both the number of iterations and the CPU which are respectively denoted by ``IT" and "CPU" in the tables.
In Tables, a dagger $(\dag)$ means that the method fails to convergence in 500 iterations.
In all the tests we use a zero vector as an initial guess and stopping criterion
$$
\frac{\parallel b-A u^{(k)} \parallel_{2}}{\parallel b\parallel_{2}}<10^{-10}
$$
is always used, where $u^{(k)}=x^{(k)}+iy^{(k)}$. In the implementation of the preconditioners to solve the inner symmetric positive definite system of linear equations we use the sparse Cholesky factorization incorporated with the symmetric approximate minimum degree reordering \cite{saad}. To do so we have used the \verb"symamd.m" command of \textsc{Matlab}. All runs are implemented in \textsc{Matlab} R2014b with a personal computer with 2.40 GHz central processing unit (Intel(R) Core(TM) i7-5500), 8 GB memory and Windows 10 operating system.

\begin{example}\label {Ex1}\rm
(see \cite{Bai1}) Consider the system of linear equations
\begin{equation}
[(K+\frac{3-\sqrt{3}}{\tau})+i\big(K+\frac{3+\sqrt{3}}{\tau}I)]x=b,
\end{equation}
where $\tau$ is the time step-size and $K$ is the five-point centered difference matrix approximating the negative Laplacian operator $L=-\Delta$ with homogeneous Dirichlet boundary conditions, on a uniform mesh in the unit square $[0, 1] \times [0, 1]$ with the mesh-size $h=1/(m + 1)$. The matrix $K\in\mathbb{R }^{n \times n}$ possesses the tensor-product form $K=I\otimes V_m + V_m \otimes I$, with $V_m=h^{-2} tridiag (−1, 2,−1)\in \mathbb{R }^{m \times m}$. Hence, $K$ is an $n \times n$ block-tridiagonal matrix, with $n = m^2$. We take
$$
W=K+\frac{3-\sqrt{3}}{\tau}I, \quad \textrm{and}  \quad T=K+\frac{3+\sqrt{3}}{\tau}I
$$
and the right-hand side vector $b$ with its jth entry $b_j$ being given by
$$
b_j=\frac{(1-i)j}{\tau(j+1)^2},  \quad  j=1,2, \ldots ,n.
$$
In our tests, we take $\tau= h$. Furthermore, we normalize coefficient matrix and right-hand side by multiplying both by $h^2$.
\end{example}

\begin{example}\label {Ex2}\rm
(See \cite{Bai1}) Consider the system of linear  equations $\eqref{Eq1} $ as following
$$
\big[ (-\omega ^2 M+K) +i(\omega C_V +C_ H) \big] =b,
$$
where $M$ and $K$ are the inertia and the stiffness matrices, $C_V$ and $C_H$ are the viscous and the hysteretic damping matrices, respectively, and $ \omega$ is the driving circular frequency. We take $C_H = \mu K$ with $\mu$ a damping coefficient, $M = I, C_V = 10I$, and $K$ the five-point centered difference matrix approximating the negative Laplacian operator with homogeneous Dirichlet boundary conditions, on a uniform mesh in the unit square $[0, 1] \times [0, 1]$ with the mesh-size $h =1/(m+1)$. The matrix $K \in \mathbb{R }^{n \times n}$ possesses the tensor-product form $K = I \otimes V_m+V_m \otimes I$, with $V_m =h^{-2} tridiag(-1, 2,-1) \in\mathbb{R }^{m \times m} $. Hence, $K$ is an $n \times n$ block-tridiagonal matrix, with $n = m^2$. In addition, we set $\mu = 8$, $\omega = \pi$, and the right-hand side vector b to be $b = (1 + i)A\textbf{1}$, with \textbf{1} being the vector of all entries equal to 1. As before, we normalize the system by multiplying both sides through by $h^2$.
\end{example}

\begin{table}[!ht]
\caption{Numerical results for  Example \ref{Ex1}.}\label{Table1}
\begin{tabular}{ccc|c|cc|cc|cc|cc|cc|cc|cc|} \hline
\multicolumn{1}{c}{Method}   &  \multicolumn{1}{c}{}   &  \multicolumn{2}{c}{$m=32$}  &  \multicolumn{2}{c}{$m=64$}  & \multicolumn{2}{c}{$m=128$}  & \multicolumn{2}{c}{$m=256$}  & \multicolumn{2}{c}{$m=512$}   & \multicolumn{2}{c}{$m=1024$}\\ [2mm] \hline

\multicolumn{1}{c}{GMRES(5)}  &  \multicolumn{1}{c}{IT} & \multicolumn{2}{c}{349}  & \multicolumn{2}{c}{\dag}   & \multicolumn{2}{c}{\dag} & \multicolumn{2}{c}{\dag}  & \multicolumn{2}{c}{\dag}  & \multicolumn{2}{c}{\dag} \\

 \multicolumn{1}{c}{}  &  \multicolumn{1}{c}{CPU} & \multicolumn{2}{c}{0.39}  & \multicolumn{2}{c}{-}   & \multicolumn{2}{c}{-} & \multicolumn{2}{c}{-}  & \multicolumn{2}{c}{-}  & \multicolumn{2}{c}{-} \\ \\

 \multicolumn{1}{c}{MHSS-GMRES(5)}  &  \multicolumn{1}{c}{$\alpha_{opt}$} & \multicolumn{2}{c}{10}  & \multicolumn{2}{c}{9.1}   & \multicolumn{2}{c}{4.7} & \multicolumn{2}{c}{5.1}  & \multicolumn{2}{c}{10.5}  & \multicolumn{2}{c}{} \\

 \multicolumn{1}{c}{}  &  \multicolumn{1}{c}{IT} & \multicolumn{2}{c}{54}  & \multicolumn{2}{c}{26}   & \multicolumn{2}{c}{71} & \multicolumn{2}{c}{114}  & \multicolumn{2}{c}{179}  & \multicolumn{2}{c}{} \\

 \multicolumn{1}{c}{}  &  \multicolumn{1}{c}{CPU} & \multicolumn{2}{c}{0.73}  & \multicolumn{2}{c}{0.12}   & \multicolumn{2}{c}{6.01} & \multicolumn{2}{c}{50.84}  & \multicolumn{2}{c}{541.41}  & \multicolumn{2}{c}{} \\   \\

 \multicolumn{1}{c}{GSOR-GMRES(5)}  &  \multicolumn{1}{c}{$\alpha_{opt}$} & \multicolumn{2}{c}{0.037}  & \multicolumn{2}{c}{0.457}   & \multicolumn{2}{c}{0.432} & \multicolumn{2}{c}{0.418}  & \multicolumn{2}{c}{0.412}  & \multicolumn{2}{c}{0.411} \\

 \multicolumn{1}{c}{}  &  \multicolumn{1}{c}{IT} & \multicolumn{2}{c}{23}  & \multicolumn{2}{c}{25}   & \multicolumn{2}{c}{26} & \multicolumn{2}{c}{26}  & \multicolumn{2}{c}{27}  & \multicolumn{2}{c}{27} \\

 \multicolumn{1}{c}{}  &  \multicolumn{1}{c}{CPU} & \multicolumn{2}{c}{0.11}  & \multicolumn{2}{c}{0.43}   & \multicolumn{2}{c}{2.57} & \multicolumn{2}{c}{14.92}  & \multicolumn{2}{c}{83.01}  & \multicolumn{2}{c}{413.93} \\   \\

 \multicolumn{1}{c}{BLTP-GMRES(5)}  &  \multicolumn{1}{c}{$\alpha_{opt}$} & \multicolumn{2}{c}{1.4}  & \multicolumn{2}{c}{1.4}   & \multicolumn{2}{c}{1.5} & \multicolumn{2}{c}{1.5}  & \multicolumn{2}{c}{1.5}  & \multicolumn{2}{c}{1.5} \\

 \multicolumn{1}{c}{}  &  \multicolumn{1}{c}{IT} & \multicolumn{2}{c}{6}  & \multicolumn{2}{c}{7}   & \multicolumn{2}{c}{7} & \multicolumn{2}{c}{7}  & \multicolumn{2}{c}{7}  & \multicolumn{2}{c}{7} \\

 \multicolumn{1}{c}{}  &  \multicolumn{1}{c}{CPU} & \multicolumn{2}{c}{0.02}  & \multicolumn{2}{c}{0.1}   & \multicolumn{2}{c}{0.59} & \multicolumn{2}{c}{3.51}  & \multicolumn{2}{c}{18.98}  & \multicolumn{2}{c}{98.37} \\   \hline

\end{tabular}
\label{Tbl1}
\caption{Numerical results for  Example \ref{Ex2}.}
\begin{tabular}{ccc|c|cc|cc|cc|cc|cc|cc|cc|} \hline

\multicolumn{1}{c}{Method}   &  \multicolumn{1}{c}{}   &  \multicolumn{2}{c}{$m=32$}  &  \multicolumn{2}{c}{$m=64$}  & \multicolumn{2}{c}{$m=128$}  & \multicolumn{2}{c}{$m=256$}  & \multicolumn{2}{c}{$m=512$}   & \multicolumn{2}{c}{$m=1024$}\\ [2mm] \hline

\multicolumn{1}{c}{GMRES(5)}  &  \multicolumn{1}{c}{IT} & \multicolumn{2}{c}{\dag}  & \multicolumn{2}{c}{\dag}   & \multicolumn{2}{c}{\dag} & \multicolumn{2}{c}{\dag}  & \multicolumn{2}{c}{\dag}  & \multicolumn{2}{c}{\dag} \\

 \multicolumn{1}{c}{}  &  \multicolumn{1}{c}{CPU} & \multicolumn{2}{c}{-}  & \multicolumn{2}{c}{-}   & \multicolumn{2}{c}{-} & \multicolumn{2}{c}{-}  & \multicolumn{2}{c}{-}  & \multicolumn{2}{c}{-} \\ \\

 \multicolumn{1}{c}{MHSS-GMRES(5)}  &  \multicolumn{1}{c}{$\alpha_{opt}$} & \multicolumn{2}{c}{81}  & \multicolumn{2}{c}{110}   & \multicolumn{2}{c}{-} & \multicolumn{2}{c}{-}  & \multicolumn{2}{c}{-}  & \multicolumn{2}{c}{-} \\

 \multicolumn{1}{c}{}  &  \multicolumn{1}{c}{IT} & \multicolumn{2}{c}{73}  & \multicolumn{2}{c}{243}   & \multicolumn{2}{c}{\dag} & \multicolumn{2}{c}{\dag}  & \multicolumn{2}{c}{\dag}  & \multicolumn{2}{c}{\dag} \\

 \multicolumn{1}{c}{}  &  \multicolumn{1}{c}{CPU} & \multicolumn{2}{c}{0.29}  & \multicolumn{2}{c}{4.65}   & \multicolumn{2}{c}{-} & \multicolumn{2}{c}{-}  & \multicolumn{2}{c}{-}  & \multicolumn{2}{c}{-} \\   \\

 \multicolumn{1}{c}{GSOR-GMRES(5)}  &  \multicolumn{1}{c}{$\alpha_{opt}$} & \multicolumn{2}{c}{0.099}  & \multicolumn{2}{c}{0.099}   & \multicolumn{2}{c}{0.099} & \multicolumn{2}{c}{0.099}  & \multicolumn{2}{c}{0.099}  & \multicolumn{2}{c}{0.099} \\

 \multicolumn{1}{c}{}  &  \multicolumn{1}{c}{IT} & \multicolumn{2}{c}{65}  & \multicolumn{2}{c}{70}   & \multicolumn{2}{c}{71} & \multicolumn{2}{c}{67}  & \multicolumn{2}{c}{63}  & \multicolumn{2}{c}{61} \\

 \multicolumn{1}{c}{}  &  \multicolumn{1}{c}{CPU} & \multicolumn{2}{c}{0.23}  & \multicolumn{2}{c}{1.07}   & \multicolumn{2}{c}{6.97} & \multicolumn{2}{c}{38.01}  & \multicolumn{2}{c}{196.88}  & \multicolumn{2}{c}{959.75} \\   \\

 \multicolumn{1}{c}{BLTP-GMRES(5)}  &  \multicolumn{1}{c}{$\alpha_{opt}$} & \multicolumn{2}{c}{0.4}  & \multicolumn{2}{c}{0.4}   & \multicolumn{2}{c}{0.4} & \multicolumn{2}{c}{0.4}  & \multicolumn{2}{c}{0.4}  & \multicolumn{2}{c}{0.4} \\

 \multicolumn{1}{c}{}  &  \multicolumn{1}{c}{IT} & \multicolumn{2}{c}{8}  & \multicolumn{2}{c}{8}   & \multicolumn{2}{c}{8} & \multicolumn{2}{c}{8}  & \multicolumn{2}{c}{8}  & \multicolumn{2}{c}{8} \\

 \multicolumn{1}{c}{}  &  \multicolumn{1}{c}{CPU} & \multicolumn{2}{c}{0.03}  & \multicolumn{2}{c}{0.11}   & \multicolumn{2}{c}{0.7} & \multicolumn{2}{c}{3.96}  & \multicolumn{2}{c}{21.99}  & \multicolumn{2}{c}{112.15} \\   \hline
\end{tabular}
\label{Tbl2}
\end{table}

\begin{table}[!ht]
\caption{Numerical results for  Example \ref{Ex3}.}
\begin{tabular}{ccc|c|cc|cc|cc|cc|cc|cc|cc|} \hline

\multicolumn{1}{c}{Method}   &  \multicolumn{1}{c}{}   &  \multicolumn{2}{c}{$m=32$}  &  \multicolumn{2}{c}{$m=64$}  & \multicolumn{2}{c}{$m=128$}  & \multicolumn{2}{c}{$m=256$}  & \multicolumn{2}{c}{$m=512$}   & \multicolumn{2}{c}{$m=1024$}\\ [2mm] \hline

\multicolumn{1}{c}{GMRES(5)}  &  \multicolumn{1}{c}{IT} & \multicolumn{2}{c}{235}  & \multicolumn{2}{c}{\dag}   & \multicolumn{2}{c}{\dag} & \multicolumn{2}{c}{\dag}  & \multicolumn{2}{c}{\dag}  & \multicolumn{2}{c}{\dag} \\

 \multicolumn{1}{c}{}  &  \multicolumn{1}{c}{CPU} & \multicolumn{2}{c}{0.41}  & \multicolumn{2}{c}{-}   & \multicolumn{2}{c}{-} & \multicolumn{2}{c}{-}  & \multicolumn{2}{c}{-}  & \multicolumn{2}{c}{-} \\ \\

 \multicolumn{1}{c}{MHSS-GMRES(5)}  &  \multicolumn{1}{c}{$\alpha_{opt}$} & \multicolumn{2}{c}{52}  & \multicolumn{2}{c}{18}   & \multicolumn{2}{c}{-} & \multicolumn{2}{c}{-}  & \multicolumn{2}{c}{-}  & \multicolumn{2}{c}{-} \\

 \multicolumn{1}{c}{}  &  \multicolumn{1}{c}{IT} & \multicolumn{2}{c}{120}  & \multicolumn{2}{c}{272}   & \multicolumn{2}{c}{\dag} & \multicolumn{2}{c}{\dag}  & \multicolumn{2}{c}{\dag}  & \multicolumn{2}{c}{\dag} \\

 \multicolumn{1}{c}{}  &  \multicolumn{1}{c}{CPU} & \multicolumn{2}{c}{0.6}  & \multicolumn{2}{c}{6.95}   & \multicolumn{2}{c}{-} & \multicolumn{2}{c}{-}  & \multicolumn{2}{c}{-}  & \multicolumn{2}{c}{-} \\   \\

 \multicolumn{1}{c}{GSOR-GMRES(5)}  &  \multicolumn{1}{c}{$\alpha_{opt}$} & \multicolumn{2}{c}{0.776}  & \multicolumn{2}{c}{0.566}   & \multicolumn{2}{c}{0.354} & \multicolumn{2}{c}{0.199}  & \multicolumn{2}{c}{0.106}  & \multicolumn{2}{c}{0.055} \\

 \multicolumn{1}{c}{}  &  \multicolumn{1}{c}{IT} & \multicolumn{2}{c}{7}  & \multicolumn{2}{c}{8}   & \multicolumn{2}{c}{11} & \multicolumn{2}{c}{22}  & \multicolumn{2}{c}{52}  & \multicolumn{2}{c}{117} \\

 \multicolumn{1}{c}{}  &  \multicolumn{1}{c}{CPU} & \multicolumn{2}{c}{0.07}  & \multicolumn{2}{c}{0.21}   & \multicolumn{2}{c}{1.76} & \multicolumn{2}{c}{20.39}  & \multicolumn{2}{c}{255.93}  & \multicolumn{2}{c}{3926.97} \\   \\

 \multicolumn{1}{c}{BLTP-GMRES(5)}  &  \multicolumn{1}{c}{$\alpha_{opt}$} & \multicolumn{2}{c}{0.4}  & \multicolumn{2}{c}{0.7}   & \multicolumn{2}{c}{1.0} & \multicolumn{2}{c}{1.4}  & \multicolumn{2}{c}{1.7}  & \multicolumn{2}{c}{2.0} \\

 \multicolumn{1}{c}{}  &  \multicolumn{1}{c}{IT} & \multicolumn{2}{c}{4}  & \multicolumn{2}{c}{5}   & \multicolumn{2}{c}{7} & \multicolumn{2}{c}{9}  & \multicolumn{2}{c}{12}  & \multicolumn{2}{c}{18} \\

 \multicolumn{1}{c}{}  &  \multicolumn{1}{c}{CPU} & \multicolumn{2}{c}{0.02}  & \multicolumn{2}{c}{0.11}   & \multicolumn{2}{c}{0.94} & \multicolumn{2}{c}{7.33}  & \multicolumn{2}{c}{55.15}  & \multicolumn{2}{c}{556.44} \\   \hline
\end{tabular}
\label{Tbl3}
\caption{Numerical results for  Example \ref{Ex4}.}
\begin{tabular}{ccc|c|cc|cc|cc|cc|cc|cc|cc|} \hline
\multicolumn{1}{c}{Method}   &  \multicolumn{1}{c}{}   &  \multicolumn{2}{c}{$m=32$}  &  \multicolumn{2}{c}{$m=64$}  & \multicolumn{2}{c}{$m=128$}  & \multicolumn{2}{c}{$m=256$}  & \multicolumn{2}{c}{$m=512$}   & \multicolumn{2}{c}{$m=1024$}\\ [2mm] \hline

\multicolumn{1}{c}{GMRES(5)}  &  \multicolumn{1}{c}{IT} & \multicolumn{2}{c}{138}  & \multicolumn{2}{c}{\dag}   & \multicolumn{2}{c}{\dag} & \multicolumn{2}{c}{\dag}  & \multicolumn{2}{c}{\dag}  & \multicolumn{2}{c}{\dag} \\

 \multicolumn{1}{c}{}  &  \multicolumn{1}{c}{CPU} & \multicolumn{2}{c}{0.15}  & \multicolumn{2}{c}{-}   & \multicolumn{2}{c}{-} & \multicolumn{2}{c}{-}  & \multicolumn{2}{c}{-}  & \multicolumn{2}{c}{-} \\ \\

 \multicolumn{1}{c}{MHSS-GMRES(5)}  &  \multicolumn{1}{c}{$\alpha_{opt}$} & \multicolumn{2}{c}{130}  & \multicolumn{2}{c}{10}   & \multicolumn{2}{c}{13} & \multicolumn{2}{c}{8}  & \multicolumn{2}{c}{-}  & \multicolumn{2}{c}{-} \\

 \multicolumn{1}{c}{}  &  \multicolumn{1}{c}{IT} & \multicolumn{2}{c}{12}  & \multicolumn{2}{c}{28}   & \multicolumn{2}{c}{84} & \multicolumn{2}{c}{283}  & \multicolumn{2}{c}{\dag}  & \multicolumn{2}{c}{\dag} \\

 \multicolumn{1}{c}{}  &  \multicolumn{1}{c}{CPU} & \multicolumn{2}{c}{0.08}  & \multicolumn{2}{c}{0.41}   & \multicolumn{2}{c}{5.75} & \multicolumn{2}{c}{88.92}  & \multicolumn{2}{c}{-}  & \multicolumn{2}{c}{-} \\   \\

 \multicolumn{1}{c}{GSOR-GMRES(5)}  &  \multicolumn{1}{c}{$\alpha_{opt}$} & \multicolumn{2}{c}{0.038}  & \multicolumn{2}{c}{0.038}   & \multicolumn{2}{c}{0.038} & \multicolumn{2}{c}{0.038}  & \multicolumn{2}{c}{0.038}  & \multicolumn{2}{c}{0.037} \\

 \multicolumn{1}{c}{}  &  \multicolumn{1}{c}{IT} & \multicolumn{2}{c}{69}  & \multicolumn{2}{c}{92}   & \multicolumn{2}{c}{75} & \multicolumn{2}{c}{66}  & \multicolumn{2}{c}{67}  & \multicolumn{2}{c}{152} \\

 \multicolumn{1}{c}{}  &  \multicolumn{1}{c}{CPU} & \multicolumn{2}{c}{0.22}  & \multicolumn{2}{c}{1.3}   & \multicolumn{2}{c}{19.08} & \multicolumn{2}{c}{194.10}  & \multicolumn{2}{c}{231.36}  & \multicolumn{2}{c}{2422.25} \\   \\

 \multicolumn{1}{c}{BLTP-GMRES(5)}  &  \multicolumn{1}{c}{$\alpha_{ opt}$} & \multicolumn{2}{c}{2.1}  & \multicolumn{2}{c}{2.2}   & \multicolumn{2}{c}{2.3} & \multicolumn{2}{c}{2.4}  & \multicolumn{2}{c}{2.5}  & \multicolumn{2}{c}{2.3} \\

 \multicolumn{1}{c}{}  &  \multicolumn{1}{c}{IT} & \multicolumn{2}{c}{21}  & \multicolumn{2}{c}{21}   & \multicolumn{2}{c}{19} & \multicolumn{2}{c}{21}  & \multicolumn{2}{c}{20}  & \multicolumn{2}{c}{20} \\

 \multicolumn{1}{c}{}  &  \multicolumn{1}{c}{CPU} & \multicolumn{2}{c}{0.06}  & \multicolumn{2}{c}{0.3}   & \multicolumn{2}{c}{9.30} & \multicolumn{2}{c}{13.20}  & \multicolumn{2}{c}{63.25}  & \multicolumn{2}{c}{369.66} \\   \hline
\end{tabular}
\label{Tbl4}
\end{table}

\begin{example}\label {Ex3}\rm
(See \cite{Bai1}) Consider the system of linear equations $\eqref{Eq1}$ as following
$$
T=I\otimes V + V \otimes I  \quad   \textrm{and}  \quad  W=10(I \otimes V_c + V_c \otimes I) +9(e_1 e^T_m + e_m (e^T_m) \otimes I,
$$
where $V = tridiag(-1, 2,-1) \in \mathbb{R}^{m \times m}$, $V_c = V - e_1 e^T_m - e_m e^T_1 \in \mathbb{R}^{m \times m}$ and $e_1$ and $e_m$ are the first and last unit vectors in $R^m$, respectively. We take the right-hand side vector $b$ to be $b = (1 + i)A\textbf{1}$, with \textbf{1} being the vector of all entries equal to 1.  Here $T$ and $W$ correspond to the five-point centered difference matrices approximating the negative Laplacian operator with homogeneous Dirichlet boundary conditions and periodic boundary conditions, respectively, on a uniform mesh in the unit square $[0, 1]\in [0, 1]$ with the mesh-size
$h = 1/(m+ 1)$.
\end{example}

\begin{example}\label {Ex4}\rm
(See \cite{Bai1,Yang}) We consider the complex Helmholtz equation
$$
-\Delta u + \sigma_1 u +i \sigma_2 u = f,
$$
where $\sigma_1$ and $\sigma_2$ are real coefficient functions, $u$ satisfies Dirichlet boundary conditions in $D = [0, 1] \times [0, 1]$ and $i = \sqrt{-1}$. We discretize the problem with finite differences on a $m \times m$ grid with mesh size $h = {1}/{(m+ 1)}$. This leads to a system of linear equations
$$
((K+ \sigma_1 I) + i\sigma_2 I) x = b,
$$
where $K = I \otimes V_m + V_m \otimes I$ is the discretization of $-\Delta$ by means of centered differences, wherein $V_m = h^{-2}tridiag(-1, 2,-1) \in \mathbb{R}^{m\in m}$. The right-hand side vector $b$ is taken to be $b = (1+i)A\textbf{1}$, with \textbf{1} being the vector of all entries equal to 1. Furthermore, before solving
the system we normalize the coefficient matrix and the right-hand side vector by multiplying both by $h^2$. For the numerical tests we set $\sigma_1 =-10$ and $\sigma_2= 500$.
\end{example}

Numerical results for Examples  \ref{Ex1}-\ref{Ex4} are listed in Tables \ref{Tbl1}-\ref{Tbl4}, respectively.
We use GMRES(5) as the iterative method. In the tables, the numerical results of the GMRES(5) method in conjunction with the MHSS, the GSOR and the BLT preconditioners are denoted by  MHSS-GMRES(5), GSOR-GMRES(5) and BLTP-GMRES(5). For the BLT and MHSS preconditioners the optimal value of $\alpha$ ($\alpha_{opt}$) were found experimentally and are the ones resulting in the least numbers of iterations. For the GSOR preconditioner the values was computed by using Theorem 2 in \cite{Salkuyeh1}. As the numerical results show for all the four examples the preconditioned GMRES(5) with the BLT preconditioner outperforms the other methods in terms of the number of the iterations and the CPU time.

Surprisingly, we see from Tables \ref{Tbl1}, \ref{Tbl2} and \ref{Tbl3} that the number of iterations of the BLT preconditioner, in contrast to the other methods,  does not grow with the problem size. However, this growth for Example \ref{Ex3} is slow for the BLT preconditioner. Another observation which can be posed here is that, for the Examples \ref{Ex1}, \ref{Ex2} and \ref{Ex3}, the optimal value of the parameter $\alpha$ is not sensitive to the problem size.

\section{Conclusion}\label{Sec4}

We have established and analyzed a block lower triangular (BLT) preconditioner to expedite the convergence speed of  the Krylov subspace methods such as GMRES, for solving an important class of complex symmetric system of  linear equations. Eigenvalues distribution of the preconditioned matrix has been intestigated.
Since the proposed preconditioner involves a parameter, we have defined an interval for the choosing a suitable parameter. We have compared the numerical results of the GMRES(5) in conjunction with BLT preconditioner with those of the GSOR and MHSS preconditioners. Numerical results show that the BLT preconditioner is superior to the other methods in terms of both iteration number and CPU time.

\section*{Acknowledgements}

The work of Davod Khojasteh Salkuyeh is partially supported by University of Guilan.

\end{document}